\documentclass{amsart}
\usepackage{amsmath,amsthm}
\usepackage{amsfonts,amssymb}
\usepackage{accents}
\usepackage{enumerate}
\usepackage{accents,color}
\usepackage{graphicx}
\usepackage{todonotes}

\newcommand{\lvt}{\left|\kern-1.35pt\left|\kern-1.3pt\left|}
\newcommand{\rvt}{\right|\kern-1.3pt\right|\kern-1.35pt\right|}

\newtheorem{thm}{Theorem}[section]

\newtheorem{lem}[thm]{Lemma}
\newtheorem{prop}[thm]{Proposition}

\theoremstyle{remark}

 \def\la{{\langle}}
 \def\ra{{\rangle}}

 \def\d{\mathrm{d}}
 
 \def\sph{{\mathbb{S}^{d-1}}}

 \def\a{{\alpha}}
 \def\b{{\beta}}

 \def\t{{\theta}}
 \def\l{{\lambda}}
 \def\o{{\omega}}
 \def\s{\sigma}
 \def\la{{\langle}}
 \def\ra{{\rangle}}

 \def\CD{{\mathcal D}}
 
 \def\CH{{\mathcal H}}

 \def\CL{{\mathcal L}}

 \def\CV{{\mathcal V}}
 
 \def\CW{{\mathcal W}}
 \def\BB{{\mathbb B}}

 \def\RR{{\mathbb R}}
 \def\SS{{\mathbb S}}
 \def\VV{{\mathbb V}}

\def\lla{\langle{\kern-2.5pt}\langle}      
\def\rra{\rangle{\kern-2.5pt}\rangle}

\newcommand{\wh}{\widehat}

\def\f{\frac}

\def\quotation#1{
\smallskip
\hskip 0.2in{\vbox{
\hsize=4.5in
\it #1}}
\smallskip}

\graphicspath{{./}}
\begin{document}

\title{non-homogeneous wave equation on a cone}

\author{Sheehan Olver}
\address{Department of Mathematics\\
Imperial College\\
 London \\
 United Kingdom  }\email{s.olver@imperial.ac.uk}

\author{Yuan Xu}
\address{Department of Mathematics\\ University of Oregon\\
    Eugene, Oregon 97403-1222.}\email{yuan@uoregon.edu}

 
\date{\today}

\begin{abstract} 
The wave equation $\left(\partial_{tt} - c^2 \Delta_x\right) u(x,t) = e^{-t} f(x,t)$  in the cone 
$\{(x,t): \|x\| \le t, x\in \RR^d, t \in \RR_+\}$ is shown to have a unique solution if $u$ and its 
partial derivatives in $x$ are in $L^2(e^{-t})$ on the cone, and the solution can be explicit given in 
the Fourier series of orthogonal polynomials on the cone. This provides a particular solution 
for the boundary value problems of the non-homogeneous wave equation on the cone, 
which can be combined with a solution to the homogeneous wave equation in the cone to 
obtain the full solution. 
\end{abstract}
\maketitle
 
\section{Introduction}
\setcounter{equation}{0}
 
For a fixed constant $c > 0$, let $\VV^{d+1}$ be the cone in $\RR^{d+1}$ defined by 
$$
   \VV^{d+1} = \{ (x,t):  \|x\|^2 \le  c^2 t^2, \, x \in \RR^d, t \in \RR_+\}.  
$$
Our main result shows that a  solution for the non-homogeneous wave equation
\begin{equation}\label{eq:wave-eqn}
        \left(\partial_{tt} - c^2 \Delta_x\right) U(x,t) = e^{-t} f(x,t),
\end{equation} 
can be given explicitly in terms of an orthogonal series for a fairly generic function $f$, and this
solution is unique if $U(x,t) = e^{-t} u(x,t)$ and $u$ satisfies a fairly mild condition. 

To describe our main result, we define the inner product, for $\mu > -1$,  
$$
  \la f,g\ra_\mu: =  b_\mu \int_{\VV^{d+1}} f(x,t) g(x,t) W_\mu(x,t)\d x \d t, \qquad  
      W_\mu(x,t)= (c^2 t^2-\|x\|^2)^\mu e^{-t},
$$
where $b_\mu$ is the constant chosen so that $\la 1,1\ra =1$. A basis of orthogonal polynomials with respect to
this inner product is given in \cite{X} in terms of the Laguerre and Jacobi polynomials and spherical harmonics. 
More precisely, let 
\begin{equation}\label{eq:basisCone2}
  Q_{m, j, \ell}^{n,\mu}(x,t) = L_{n-m}^{2 m + 2\mu + d}(t) t^{2j}  P_j^{(\mu, m-2j+\f{d-2}{2})} 
       \left(2 \frac{\|x\|^2}{c^2 t^2}-1\right) Y_{\ell}^{m-2j}(x), 
\end{equation}
where $\{Y_\ell^{m-2j}\}$  is an orthonormal basis of spherical harmonics of degree $m-2j$ of $d$-variables, 
$0\le j \le m/2$ and $0\le m \le n$. Then $\{Q_{m,j,\ell}^{n,\mu}\}$ is an orthogonal basis of $L^2(\VV^{d+1}, W_\mu)$. 
In particular, the Fourier orthogonal series of $f\in L^2(\VV^{d+1},e^{-t})$ is given by 
$$
    f = \sum_{n=0}^\infty \sum_{m=0}^n \sum_{j=0}^{ \lfloor \f m 2 \rfloor} \sum_{\ell} 
        \wh f_{m, j, \ell}^n Q_{m,j,\ell}^{n,0}, \qquad 
          \wh f_{m, j, \ell}^n = \frac{ \la f, Q_{m,j,\ell}^{n,0}\ra_0}{\la Q_{m,j,\ell}^{n,0}, Q_{m,j,\ell}^{n,0}\ra_0}.
$$ 
The polynomials in \eqref{eq:basisCone2} are in fact well defined when $\mu = -1$ (see Section 2 for details). 
We can now state our main result. 

\begin{thm}\label{thm:mainU}
Let $f \in L^2(\VV^{d+1}, e^{-t})$ be a smooth function. The wave equation 
\begin{equation}\label{eq:wave-eqn1}
        \left(\partial_{tt} - c^2 \Delta_x\right) U(x,t) = e^{-t} f(x,t), \\
\end{equation} 
has a solution $U(x,t) = e^{-t} u(x,t)$, where $u$ is given by 
\begin{align} \label{eq:solu-wave}
   u(x,t) \,& = \sum_{n=0}^\infty \sum_{m=0}^n \sum_{j=0}^{ \lfloor \f m 2 \rfloor}\sum_\ell 
      u_{m,j,\ell}^n Q _{m,j,\ell}^{n,-1}(x,t)
 \end{align}
with the coefficients $u_{m,j,\ell}^n$ determined by  
\begin{equation*}
  u_{m,j,\ell}^n = \frac{1}{a_{m,j}} \sum_{i=0}^{\lfloor \frac{n-m}{2} \rfloor} \frac{(n-m)!}{(n-m-2i)!} 
      \wh f_{m+2i,j+i,\ell}^n,
\end{equation*}
where $a_{m,j} = (2m-2j+d-2)/(2m+d-2)$ and we assume $a_{0,0} = 1$ for $d=2$.
Furthermore, the solution is unique if $u$ and $\partial_{x_j} u$ are  in $L^2 (\VV^{d+1}, e^{-t})$. 
\end{thm}

The precise condition on the smoothness of $f$ will be given in Section 3, together with further discussion of
our main results and examples. The uniqueness of the solution holds because of the uniform decaying in 
$t$ variable in $u \in L^2(\VV^{d+1}, e^{-t})$. For example, if $d =3$, then all smooth functions of the form
$$
       U(x,t) = \frac{1}{r} \left[f_1(t-r) + f_2(t+r)\right], \qquad r = \|x\|,
$$ 
satisfy the wave equation $(\partial_{tt} - \Delta_x) U =0$. These functions, however, do not satisfy $e^{t} U(x,t)
 \in W_2^1(\VV^{d+1})$ if $f\ne 0$. 

Our approach is motivated by spectral methods for solving partial differential equations on the unit ball. 
Classical orthogonal polynomials on the unit ball are orthogonal with respect to the weight function
$\varpi_\mu(x) = (1-\|x\|^2)^\mu$ for $\mu > -1$. Let $\CV_n^d(\varpi_\mu)$ denote the space of 
orthogonal polynomials with respect to $\varpi_\mu$ on the unit ball $\BB^d$. An orthogonal basis,
denoted by $P_{j,\ell}^{n,\mu}$ of $\CV_n^d(\varpi_\mu)$ can be given explicitly in terms of the Jacobi
polynomials and spherical harmonics. These polynomials can be extended to $\mu  =-1$ and 
enjoy an orthogonality with respect to an inner product that contains derivatives \cite{X08}. Spectral-Galerkin 
methods based on orthogonal polynomials have been used to solve the Laplace equation or Helmholtz 
equation on the unit disk or the unit ball \cite{ACH1, ACH2, Boyd, LX, Vasil}. As it is shown in \cite{LX}, the Laplace 
operator maps the space $\CV_n^d(\varpi_{-1})$ into $\CV_n(\varpi_0)$. It turns out that elements of 
$\CV_n^d(\varpi_{-1})$ are orthogonal polynomials with respect to an inner product that involves derivatives,
which arises naturally from the weak formulation of the differential equation.  
It is this analogy that leads us to consider the action of thes wave operator on the polynomials 
$Q_{m,j,\ell}^{n, -1}$, which turns out to be an element in $\CV_n(\VV^{d+1}, e^{-t})$ and satisfies 
a surprisingly simple formula. The latter leads us to our main result. 

There are a number of implications of these results in the numerical solution of wave problems.  Typically wave problems are solved numerically by truncation of the domain, at which point one must deal with the introduction of  boundary conditions at the artificially introduced boundary, for example, perfectly matched layers \cite{B}. To quote \cite{Z}\footnote{The reference numbers have been changed to refer to those in this paper.}:

\smallskip
\quotation{There has been significant developments in the treatment of artificial outer boundaries since the 70s, but there is no consensus on an optimal method \cite{G,H}. Especially the construction of boundary conditions for nonlinear problems is difficult \cite{S}.}
\smallskip

\noindent In \cite{Z}, it is advocated to solve hyperbolic PDEs on unbounded space-time geometries by {\it compactification} of the geometry by a coordinate transformation. Instead, our solution approach indicates that it is possible to solve directly, and in closed form, on the unbounded geometry. As the solution cannot propagate faster than the wave speed $c$, we are guaranteed to capture the global solution of any inhomogenous wave problem with compact supported forcing term, provided it can be well-represented in the proposed orthogonal polynomial basis. This solution technique would serve as a useful preconditioner for variable-coefficient wave problems, as in those that arise in iterative solutions of nonlinear wave problems. We note that an efficient Laguerre transform would be required to make this feasible; we mention recent progress on fast Jacobi and spherical harmonics transforms based on recurrence relationships of orthogonal polynomials \cite{TWO,SlJac,Sz} which may be adaptable to this setting.  

One is left with the task of solving the homogeneous wave equation with boundary conditions in the cone. 
In one-spatial dimension this can be accomplished trivially via a traveling wave solution constructed in terms of the boundary data. We leave the higher-dimensional problem as an open problem, though note that the unified transform method has proven useful for deriving explicit integral solutions for linear PDEs in convex geometries \cite{DTV,F}, including the wave equation on the half-plane \cite{DGSV}, and so may be applicable here. 

The paper is organized as follows.  In Section~\ref{sec:OPs}, we consider the orthogonal structure on the cone
and establish the action of wave operators on the cone. The main result is discussed and proved in Section~\ref{sec:main}.

\medskip
\noindent
{\bf Acknowledgment}. The  authors would like to thank the Isaac Newton Institute for Mathematical 
Sciences, Cambridge, for support and hospitality during the programmes ``Approximation, sampling and 
compression in data science" and ``Complex analysis: techniques, applications and computations'',  where work on this paper was undertaken. This work is supported by EPSRC 
grant no EP/K032208/1.

\section{Orthogonal polynomials on the cone}\label{sec:OPs}
\setcounter{equation}{0} 

We discuss orthogonal polynomials on the cone in the first subsection and the action of the wave operator 
on a family of polynomials in the second subsection. 

Throughout this section we assume $c =1$, so that $W_\mu(x,t)  =(t^2-\|x\|^2)^\mu e^{-t}$ and the 
cone is defined by $\VV^{d+1} = \{(x,t): \|x\| \le t, \, x \in \RR^d, t\in \RR_+\}$. The case $c \ne 1$ can
be obtained by a simple dilation $x \to x/c$. 

\subsection{Cone polynomials}
Let $\CV_n(\VV^{d+1},W_{\mu})$ denote the space of orthogonal polynomials of degree $n$ with 
respect to the inner product $\la \cdot,\cdot\ra_\mu$ in the introduction. An orthognal  basis of this space 
can be given in terms of the Laguerre polynomials and orthogonal polynomials on the unit ball. 
First we recall the definition of several families of orthogonal polynomials. 

The Laguerre polynomials $L_n^\a$ are orthogonal with respect to the weight function $t^\a e^{-t}$ on 
$\RR_+ = [0,\infty)$ and it is given by 
$$
  L_n^\a(t) = \frac{(\a+1)_n}{n!} \sum_{k=0}^n \frac{(-n)_k}{k! (\a+1)_k} t^k,
$$ 
where $(a)_k = a(a+1)\ldots(a+k-1)$ is the Pochhammer symbol. Its $L^2$ norm is  
$$
h_n^\a  = \frac{1}{\Gamma(\a+1)}  \int_0^\infty [L_n^\a(t)]^2t^\a e^{-t} \d t = \frac{(\a+1)_n}{n!}.
$$
The Jacobi polynomials $P_n^{(\a,\b)}$ are orthogonal with respect to the weight function $w_{\a,\b}(t)
= (1-t)^\a(1+t)^\b$ on $[-1,1]$ and, in terms of the Gauss hypergeometric function, 
$$
  P_n^{(\a,\b)}(t) = \frac{(\a+1)_n}{n!} {}_2F_1\left(\begin{matrix} -n, n+ \a + \b+1 \\ \a+1\end{matrix}; \frac{1-t}{2}\right).
$$
With normalization constant $c_{\a,\b} = 1/\int_{-1}^1 w_{\a,\b}(t) \d t$, its $L^2$ norm is given by 
$$
  h_n^{(\a,\b)}= c_{\a,\b}\int_{-1}^1 P_n^{(\a,\b)}(t)P_m^{(\a,\b)}(t) w_{\a,\b} (t) \d t = \frac{(\a+1)_n (\b+1)_n(\a+\b+n+1)}{n!(\a+\b+2)_n(\a+\b+n+1)}.
$$

Let $\CV_n^d(\varpi_\mu)$ be the space of orthogonal polynomials with respect to
$$
    \varpi_\mu(x) = (1-\|x\|^2)^\mu, \qquad \mu > -1,
$$ 
on the unit ball $\BB^d$. When $d =1$, the orthogonal polynomial is the Gegenbauer polynomial $C_n^{\mu+\f12}$.
For $d > 1$, these orthogonal polynomials are also classical, see \cite[Section 5.2]{DX}. In particular,
a basis of $\CV_n^d(\varpi_\mu)$ can be given in terms of the Jacobi polynomials and spherical harmonics.
A spherical harmonic $Y$ of $d$-variables is a homogeneous polynomial that satisfies $\Delta Y = 0$, 
where $\Delta$ is the Laplace operator of $\RR^d$. They are orthogonal on the unit sphere with respect to 
the inner product
$$
  \la f,g\ra_{\sph} = \frac{1}{\o_d} \int_{\sph} f(\xi)g(\xi) \d \s(\xi),
$$ 
where $\o_d$ is the surface area of $\sph$ and $\d\s$ is the Lebesgue measure on the sphere. For 
$m =0,1,2,\ldots$, let $\CH_m^d$ be the space of spherical harmonics of degree $n$ in $d$ variables. 
It is know that $\dim \CH_n^d = \binom{n+d-1}{d} = \binom{n+d-3}{d-1}$. For $0 \le j \le m/2$, let  
$\{Y_{\ell}^{m-2j}:  1 \le \ell \le \dim \CH_{m-2j}^d\}$ be an orthonormal basis of $\CH_{m-2j}^d$. Define 
\begin{equation}\label{eq:basisBall}
 P_{j,\ell}^m(\varpi_\mu;x) := P_j^{(\mu, m-2j+\f{d-2}{2})} (2 \|x\|^2-1) Y_{\ell}^{m-2j}(x), \qquad x \in \BB^d.
\end{equation}
Then $\{ P_{j,\ell}^m(\varpi_\mu; \cdot): 0\le j \le m/2, \,\, 1 \le \ell \le \dim \CH_n^d\}$ is an orthogonal 
basis of $\CV_n^d(\varpi_\mu)$ \cite[(5.2.4)]{DX} for $d > 1$. 
 
We are now ready to define our orthogonal basis of polynomials for the space $\CV_n(\VV^{d+1},W_{\mu})$ 
on the cone. For $0 \le m \le n$, define
\begin{equation}\label{eq:basisCone}
  Q_{m, j, \ell}^{n,\mu}(x,t) = L_{n-m}^{2 m + 2\mu + d}(t) t^m P_{j,\ell}^m \left(\varpi_\mu;\frac{x}{t}\right), 
\end{equation}
where $P_{j,\ell}^n$ is the basis defined in \eqref{eq:basisBall} for $d > 1$ and $C_m^{\mu+\f12}$ when
$d =1$. Using the fact that $Y_\ell^{n-2m}$ is 
homogeneous, we see that this is the same as the expression in \eqref{eq:basisCone2}. The $L^2$ 
norm of the polynomial can be deduced from the norms of the Laguerre and the Jacobi polynomials, 
which gives
\begin{align}
    b_\mu \int_{\VV^{d+1}} & |Q_{m,j,\ell}^{n,\mu}(x,t)|^2 W_\mu(x,t) \d x \d t  \\
        & = \frac{(2\mu+d+1)_{n+m} 
              (\mu+1)_j (\frac{d}{2})_{m-j}(\mu+m-j+\f{d}{2})} {(n-m)! j! (\mu+ \frac{d}{2}+1)_{m-j} (\mu+m+\f{d}{2})}. \notag
\end{align}
In particular, for $\mu =0$, we obtain
\begin{equation}\label{eq:hnm-norm}
  h_{m,n}: = b_0 \int_{\VV^{d+1}} |Q_{m,j,\ell}^{n,0}(x,t)|^2 e^{-t} \d x \d t = \frac{d (d+1)_{m+n}}{(2 m+d)(n-m)!}. 
\end{equation}
It should be mentioned that these polynomials are eigenfunctions of a second order differential operator
\cite{X}: Every $u \in \CV_n(\VV^{d+1}, W_{\mu})$ satisfies the differential equation
\begin{align}\label{eq:cone-eigenL}
    \left[ t \left( \Delta_x+ \partial_t^2 \right) + 2  \la x,\nabla_x \ra \partial_t   - \la x,\nabla_x\ra 
        + (2\mu + d +1 -t) \partial_t \right] u = - n u.
  \end{align} 

Analytic continuation shows that the polynomials $Q_{m,j,\ell}^{n,\mu}$ are well-defined if $\mu < -1$ and $\mu$ 
is not an integer. If $\mu = - k$ is a negative integer, then these polynomials are well defined for $j \ge k$ by 
using the formula \cite[(4.22.2)]{Sz},
\begin{equation}\label{eq:Jacobi-negative}
 \binom{n}{k} P_j^{(-k,\b)}(s) =  \binom{j+\b}{k} \left( \frac{s-1}{2}\right)^k P_{j-k}^{(k,\b)}(s), \quad 1 \le k \le j.
\end{equation}
Particularly important for our purpose is that $Q_{m,j,\ell}^{n, -1}$ is well defined for all $j$ if we 
define $P_0^{(-1,\b)}(t) =1$. Indeed, using \eqref{eq:Jacobi-negative} and setting $\b_j:= m-2j+\f{d-2}{2}$, 
we have
\begin{align} \label{eq:basis-1}
\begin{split}
  Q_{m, 0, \ell}^{n,-1}(x,t) &\, = L_{n-m}^{2 m + d-2}(t) Y_{\ell}^{m}(x), \\
   Q_{m, j, \ell}^{n,-1}(x,t) &\, = \frac{j+\b_j}{j} 
     L_{n-m}^{2 m + d-2}(t) t^{2j-2}  \\ 
      &\, \quad   \times (t^2-\|x\|^2)P_{j-1}^{(1,\b_j)} \left(2 \frac{\|x\|^2}{t^2} -1\right)Y_{\ell}^{m-2j}(x), \quad j \ge 1.
\end{split}
\end{align} 
These polynomials are no longer orthogonal with respect to a weight function in the $L^2$ sense, but they
are orthogonal in a Sobolev space. Let us define
$$
  \la f,g\ra_\nabla := \int_{\VV^{d+1}} \nabla_x f(x,t) \cdot\nabla_x g(x,t) \d x e^{-t} \d t 
     + \lambda \int_{\VV_0^{d+1}} f(x,t) g(x,t) t^{-1} e^{-t} \d \s(x,t), 
$$
where $\nabla_x$ is the gradient in $x$ variables, $\l > 0$ is a constant and $\d \s$ is the surface measure 
on the surface $\VV_0^{d+1}$ of the cone. It is easy to see that this is an inner product in the Sobolev space 
\begin{equation} \label{eq:W21}
W_2^1(\VV^{d+1}): = \left\{f \in L^2(\VV^{d+1}, e^{-t}): \int_{\VV^{d+1} } |\nabla_x f|^2 e^{-t} \d x \d t < \infty\right \}.
\end{equation} 

\begin{prop} \label{prop:SOP}
The polynomials $Q_{m,j,\ell}^{n,-1}$ are mutually orthogonal with respect to the inner product 
$\la \cdot,\cdot\ra_\nabla$.  
\end{prop} 

\begin{proof}
By \cite{X08}, the polynomial $Q_{m,j,\ell}^{n,-1}$ can be written as 
$$
 Q_{m, j, \ell}^{n,-1}(x,t) =  L_{n-m}^{2 m + d-2}(t) t^m P_{j,\ell}^{m,-1}\left(\frac{x}{t}\right),
$$
where $P_{j,\ell}^{m,-1}$ are mutually orthogonal polynomials with respect to the inner product 
$$
  \la f,g \ra_{\BB}:= \int_{\BB^d} \nabla f(x) \cdot\nabla g(x) \d x  
       + \lambda \int_{\sph} f(\xi) g(\xi) \d \s(\xi) 
$$
on the unit ball. Changing variable $x = t y$, so that $\nabla_x  = t^{-1} \nabla_y$, it follows that 
\begin{align*}
  \la Q_{m, j, \ell}^{n,-1},Q_{m', j', \ell'}^{n',-1} \ra_{\nabla} &  = \int_0^\infty t^{m+m'+d-2}
        L_{n-m}^{2m+d-2}(t) L_{n'-m'}^{2m'+d-2}(t) e^{-t} \d t  \la P_{j,\ell}^{m,-1} P_{j',\ell'}^{m',-1}\ra_{\BB},
\end{align*}
from which the orthogonality of $Q_{m,j,\ell}^{n,-1}$ follows readily. 
\end{proof}

We can also writing the polynomial $Q_{m,j,\ell}^{n,-1}$ as a sum of $Q_{m,j,\ell}^{k,0}$, which turns out to 
have a fairly simple form of merely six terms with rational coefficients.  Let $a_{m,j}$ be defined as in \eqref{eq:amj}. 

\begin{prop} \label{prop:Q-1toQ0}
For $0\le j \le m/2$ and $0 \le m \le n$, 
\begin{align*}
  Q_{m,j,\ell}^{n,-1} = a_{m,j} & \left[\big(Q_{m,j,\ell}^{n,0} - b_{m,n} Q_{m-2,j-1,\ell}^{n,0}\big) 
    - 2 \big( Q_{m,j,\ell}^{n-1,0} -c_{m,n}  Q_{m-2,j-1,\ell}^{n-1,0} \big)  \right.  \\
   \qquad     & \, \left. + \big(Q_{m,j,\ell}^{n-2,0} - d_{m,n}Q_{m-2,j-1,\ell}^{n-2,0}\big)\right].
\end{align*}
where 
\begin{equation} \label{eq:amj}
  a_{m,j} =  \frac{2m-2 j+d-2}{2m+d-2} \quad \hbox{and}\quad b_{m,n} = (n-m+1)(n-m+2),
\end{equation}
and we assume $a_{0,0} =1$ when $d=2$, and 
$$
c_{m,n} = (n+m+d-2)(n-m+1)\quad \hbox{and} \quad d_{m,n} =  (n+m+d-2)(n+m+d-3).
$$
\end{prop}
 
\begin{proof}
Expanding $Q_{m,j,\ell}^{n,-1}$ in terms of $\{Q_{m,j,\ell}^{k,0}\}$, the coefficients of the expansion can be
computed from 
$$
    \frac{\la Q_{m,j,\ell}^{n,-1}, Q_{m',j',\ell}^{n',0} \ra}{h_{m',n'}} = \frac{d}{h_{m',n'}
       \Gamma(d+1)}\int_{\VV^{d+1}} Q_{m,j,\ell}^{n,-1}(x,t)Q_{m',j',\ell}^{n',0}(x,t) e^{-t} \d x \d t.
$$
By the orthogonality of $Y_\ell^{m-2j}$, we can assume $m - 2j = m' - 2j'$. In particular, $\b_j = m-2j + \f{d-2}{2}$ 
remains unchanged when $(m,j)$ is replaced by $(m',j')$. Hence, separating variables, the integral in the right-hand 
side becomes a product of two integrals of one-variable. The first one is 
$$
   d \int_0^1 P_j^{(-1,\b_j)}(2r^2-1) P_j^{(0,\b_j)}(2r^2-1) r^{2\b_j+1} \d r,
$$
which can be evaluated by setting $s= 2r^2-1$ and using \cite[(18.9.5)]{DLMF}
$$
  P_j^{(-1,\b)}(s) = \frac{j+\b}{2j+\b} \left( P_j^{(0,\b)}(s)- P_{j-1}^{(0,\b)}(s)\right).
$$
These lead to two cases, $j'=j$ or $j' = j-1$. If $j' = j$ then $m'=m$ by $m - 2j = m' - 2j'$, whereas if $j' = j-1$
then $m'=m-2$. In each case the second integral, given by
$$
  \frac{1}{\Gamma(d+1)}\int_0^\infty L_{n-m}^{2m+d-2}(t) L_{n'-m'}^{2m'+d}(t)^2 t^{m+m'+d} e^{-t} dt,
$$
can be evaluated by applying the identity 
$$
   L_k^{\a}(t) = L_k^{\a+1}(t) - L_{k-1}^{\a+1}(t)
$$ 
twice. For $m' = m$, we apply the identity for $\a = 2m+d-2$. For $m'=m-2$, we apply the identity for $\a = 2m+d-4$. 
In all cases, the integrals involved reduce to the norm of the Jacobi polynomials or the Laguerre polynomials. 
We omit the details.  
\end{proof} 

\subsection{Wave operator on the cone polynomials}

In this subsection we study the action of the wave operator on the polynomials $Q_{m,j,\ell}^{n, -1}$. We
start with the action of Laplace operator. 

\begin{lem}
For $1 \le j \le m/2$, 
\begin{equation}\label{eq:LaplaceCone}
  \Delta_x \left[Q_{m,j,\ell}^{n,-1}(x,t)\right] = -4 (m-j+\tfrac{d}{2})^2 Q_{m-2,j-1,\ell}^{n-2,1}(x,t).
\end{equation}
\end{lem}

\begin{proof}
The orthogonal polynomials $P_{j,\ell}^m(\varpi_{-1})$ satisfies \cite[Lemma 3.2]{LX}
$$
   \Delta P_{j,\ell}^m(\varpi_{-1}) = - 4(m-j+\tfrac{d-2}{2})^2 P_{j-1,\ell}^{m-2}(\varpi_1), \qquad j \ge 1. 
$$
Hence, by \eqref{eq:basisCone}, 
\begin{align*}
  \Delta_x \left[Q_{m, j,\ell}^{n,-1}(x,t)\right] = &\, L_{n-m}^{2m+d-2}(t) t^{m-2} (\Delta P_{j,\ell}^m)\left(\frac{x}{t}\right) \\
   =  &\, - 4(m-j+\tfrac{d-2}{2})^2 L_{n-m}^{2m+d-2}(t) t^{m-2} P_{j-1,\ell}^{m-2}(\varpi_1),
\end{align*}
which equals to the right-hand side of \eqref{eq:LaplaceCone} by \eqref{eq:basisCone}.
\end{proof}

Next we consider the action of $\partial_{tt}$. For $j \ge 1$, we use \eqref{eq:basis-1} and write 
\begin{equation} \label{eq:Qj>0}
e^{-t} Q_{m, j, \ell}^{n,-1}(x,t) = \frac{j+\b_j}{j}  e^{-t} L_{n-m}^{2 m + d-2}(t) H_{j-1}^{(\b_j)} (x,t)Y_{\ell}^{m-2j}(x), 
\end{equation} 
where
$$
 H_{j}^{(\b)} (x,t):= (t^2- \|x\|^2) t^{2j} P_j ^{(1,\b)} \left(2 \frac{\|x\|^2}{t^2} -1\right). 
$$
Applying the Lebnitz rule, we will work out the action of $\partial_{tt}$ by considering
\begin{align} \label{eq:waveCone1}
& e^t \partial_{tt} \left[e^{-t} L_{n-m}^{2 m + d-2}(t) H_{j-1}^{(\b_j)} (x,t) \right] =  L_{n-m}^{2 m + d}(t)H_{j-1}^{(\b_j)}(x,t) \\
      &\qquad\qquad - 2 L_{n-m}^{2 m + d-1}(t)  \frac{\partial}{\partial t} H_{j-1}^{(\b_j)}(x,t) + 
         L_{n-m}^{2 m + d-2}(t) \frac{\partial^2}{\partial t^2}H_{j-1}^{(\b_j)}(x,t), \notag
\end{align}
where we have used the fact that the derivatives of the Laguerre polynomials satisfy
$$
   e^t \frac{\d}{\d t} \left[ e^{-t} L_{n}^{\a}(t) \right] = - L_{n}^{\a+1}(t) \quad \hbox{and}\quad
    e^t \frac{\d^2}{\d t^2} \left[ e^{-t} L_{n}^{\a}(t) \right] =  L_{n}^{\a+2}(t).
$$ 

\begin{lem} \label{lem:waveCone}
For $0 \le j \le m/2$, 
\begin{align*}
  \frac{\partial}{\partial t} H_j^{(\b)} (x,t)
     = \, & 2(j+1) t^{2j+1} P_j^{(0,\b)} \left(2 \frac{r^2}{t^2} -1\right), \\
  \frac{\partial^2}{\partial t^2} H_j^{(\b)}(x,t)
    = \, &2 (j+1) (4j+2 \b + 3) t^{2j} P_j^{(0,\b)} \left(2 \frac{r^2}{t^2} -1\right) \\
       &- 4 (j+1)(j+\b+1) t^{2j} P_j^{(1,\b)} \left(2 \frac{r^2}{t^2} -1\right).
\end{align*}
\end{lem}

\begin{proof}
Setting $s = 2 \frac{r^2}{t^2} -1$ and $r = \|x\|$, we can write $H_j^{(\b)}(x,t)$ as  
$$
     H_j^{(\b)}(x,t) =  2^{j} r^{2j+2}  h_{j,\b}(s) \quad \hbox{with}\quad  h_{j,\b}(s) := (1-s)(1+s)^{- j-1} P_j^{(1,\b)}(s).  
$$ 
A quick computation shows that, with $s = 2 \frac{r^2}{t^2} -1$,
\begin{align*}
  \frac{\d}{\d t} h_{j,\b} \left(2 \frac{r^2}{t^2} -1\right) & = - \frac{2}{t} (1+s) h_{j,\b}'(s), \\
  \frac{\d^2}{\d t^2} h_{j,\b}\left(2 \frac{r^2}{t^2} -1\right) & = - \frac{2}{t^2} (1+s)\left(3 h_{j,\b}'(s)+2(1+s) h_{j,\b}''(s)\right). 
\end{align*}
Using the fact that the Jacobi polynomial satisfies (\cite[(18.9.17)]{DLMF} and \cite[(18.9.5)]{DLMF})	
\begin{align*}
 (2j+\b+1)(1- s^2)\frac{d}{d s} P_j^{(1,\b)}(s) =&\,  j\big(1-\b_j - ( 2j+\b+1) s\big) P_j^{(1,\b)}(s) \\
 & + 2 (j+1) (j + \b) P_{j-1}^{(1,\b)}(s), \\
  (2j+\b+1) P_j^{(0,\b)}(s) = & \, (j+\b+1)P_j^{(1,\b_j)}(s) - (j+\b) P_{j-1}^{(1,\b)}(s),
\end{align*}
we can deduce that 
\begin{align*}
   h_{j,m}'(s) = -2 (j+1) (1+s)^{- j - 2} P_j^{(0,\b)}(s).
\end{align*}
Taking one more derivative and using (\cite[(18.9.15)]{DLMF} and \cite[(18.9.6)]{DLMF})
\begin{align*}
     \frac{\d}{\d s}  P_n^{(\a,\b)}(s) &\,= \frac{n+\a+\b+1}{2}  P_{n-1}^{(\a+1,\b+1)}(s), \\
    (n+\tfrac{\a+\b}2 +1) (1+s) P_n^{(\a,\b+1)}(s) &\, = (n+1) P_n^{(\a,\b)}(s) + (n+\b+1)P_n^{(\a,\b)}(s),
\end{align*}
we can also deduce that
\begin{align*}
  h_{j,m}''(s)  = 2 (j+1) (1 + s)^{- j-3}  \left( (2j+\b +3) P_j^{(0,\b_j)}(s) - (j+\b+1) P_j^{(1,\b_j)}(s) \right). 
\end{align*}
Putting these together proves the stated identities.  
\end{proof}

In the following, we adopt the convention that $Q_{m,j,\ell}^{n,-1}(x,t) =0$ whenever the index $m$ or $j$ is 
a negative integer. 

\begin{thm}
For $0 \le m \le n$ and $0 \le j \le m/2$, 
\begin{align} \label{eq:waveCone}
  e^t   \big(\partial_{tt} - \Delta_x \big) \left[ e^{-t}  Q_{m,j,\ell}^{n,-1}(x,t)\right] 
        =  a_{m,j} \left(Q_{m,j,\ell}^{n, 0}(x,t) - b_{m,n} Q_{m-2,j-1,\ell}^{n, 0}(x,t) \right),  
\end{align}
where $a_{m,j}$ and $b_{m,n}$ are defined in \eqref{eq:amj}.
\end{thm} 

\begin{proof}
We prove the case $j \ge 1$ by first carrying out the computation of the right-hand side of \eqref{eq:waveCone1}. 
Using \cite[(18.9.6)]{DLMF} 
$$
  (1-s) P_{j-1}^{(1,\b_j)}(s) = \frac{2 j}{m+\f{d-2}{2}} \left(P_{j-1}^{(0,\b_j)}(s) -P_j^{(0,\b_j)}(s) \right),
$$
the first term in the right-hand side of  \eqref{eq:waveCone1} is 
$$
L_{n-m}^{2 m + d}(t)H_{j-1}^{(\b_j)}(x,t)=
  \frac{2 j}{2m+d-2} t^2 L_{n-m}^{2m+d}(t) t^{2j-2} \left(P_{j-1}^{(0,\b_j)}(s) - P_j^{(0,\b_j)}(s) \right).
$$
Using Lemma \ref{lem:waveCone}, the second term in the right-hand side of \eqref{eq:waveCone1} 
becomes 
$$
     - 2 L_{n-m}^{2 m + d-1}(t)  \frac{\partial}{\partial t} H_{j-1}^{(\b_j)}(x,t)   
       =   -4 j L_{n-m}^{2 m + d-1}(t)  t^{2j-1} P_{j-1}^{(0,\b_j)}(s),
$$
and the third term in the right-hand side of \eqref{eq:waveCone1} gives 
\begin{align*}
       L_{n-m}^{2 m + d-2}(t) \frac{\partial^2}{\partial t^2}H_{j-1}^{(\b_j)}(x,t) 
          = \, &  2 j (2m+d-3) L_{n-m}^{2 m + d-2}(t) t^{2j-2} P_j^{(0,\b)}(s) \\
       &- 4 j (m-j+\tfrac{d-2}{2}) L_{n-m}^{2 m + d-2}(t) t^{2j-2} P_{j-1}^{(1,\b_j)} (s).
\end{align*}
Collecting the terms involving $P_j^{(0,\b_j)}$ in the right-hand of \eqref{eq:waveCone1} and using the identity
$$
    t^2  L_n^{\a+2}(t) - 2 \a t L_n^{\a + 1}(t) +  \a (\a - 1) L_n^\a(t) = (n + 1) (n + 2) L_{n + 2}^{\a - 2}(t),
$$
which can be easily verified using the explicit formula of the Laguerre polynomials, we obtain
\begin{align*}
   &  \frac{j+\b_j}{j} e^t \partial_t^2  \left[e^{-t}  L_{n-m}^{2 m + d-2}(t) H_{j-1}^{(\b_j)} (x,t) \right]  =  
    - a_{m,j} b_{m,n} L_{n-m}^{2m+d}(t) t^{2j+2} P_{j+1}^{(0,\b_j)}(s)\\
     & \quad + a_{m,j} L_{n-m-2}^{2 m + d}(t) t^{2j} P_j^{(0,\b_j)}(s)
           -  4 (j+\b_j) (m-j+\tfrac{d-2}{2})  L_{n-m}^{2 m + d-2}(t) t^{2j-2} P_{j-1}^{(1,\b_j)}(s),
\end{align*}
where $a_{m,j}$ and $b_{m,n}$ are as defined in \eqref{eq:amj}. Multiplying by $Y_{\ell}^{n-2j}(x)$ and using 
the definition of $Q_{m,j,\ell}^{n,\mu}$, we conclude that 
\begin{align*}
   \frac{j+\b_j}{j}  e^t \partial_t^2 (t^2-\|x\|^2)e^{-t} Q_{m,j,\ell}^{n,1}(x,t)  = & \, a_{m,j} Q_{m,j,\ell}^{n, 0}(x,t) 
        - a_{m,j} b_{m,n} Q_{m-2,j-1,\ell}^{n, 0}(x,t) \\ 
  &    - 4 (m-j+\tfrac{d-2}{2})^2 Q_{m-2,j-1,\ell}^{n-2, 1}(x,t). 
\end{align*}
The last term in the right-hand side is the action of the Laplace operator by \eqref{eq:LaplaceCone}. Hence,
by \eqref{eq:Qj>0} and \eqref{eq:waveCone1}, we have proved the case of $j \ge 1$. 

The case of $j =0$ is simple, since $\Delta Y_\ell^m =0$ so that 
$$
e^t  \big(\partial_t^2 - \Delta_x \big) \left[ e^{-t}  Q_{m,0,\ell}^{n,-1}(x,t)\right] = L_{n-m}^{2m+d}(t)Y_\ell^m(x).
$$
Since $P_{0,\ell}^m(x) = Y_\ell^m(x)$ by \eqref{eq:basisBall}, we see that the right-hand side is exactly
$Q_{m,0,\ell}^{n,1}$ by \eqref{eq:basisCone}. By $a_{m,0} =1$, this proves \eqref{eq:waveCone}
for $j =0$. The proof is completed. 
\end{proof}
 
\section{Wave equation}\label{sec:main}
\setcounter{equation}{0}

If $u(x,t)$ is the solution of the wave equation $(\partial_{tt} - \Delta_x) u = f$, then the function
$U(x,t) = u(\frac{x}{c},t)$ satisfies the wave equation $(\partial_{tt} - c^2 \Delta_x) U = F$ where 
$F(x,t)= f(\frac{x}{c}, t)$. Hence, we shall state and work with the case $c =1$.  

Recall that $W_0(x,t) = e^{-t}$ and the inner product in $L^2(\VV^{d+1}, e^{-t})$ is 
$$
     \la f,g\ra_0 = b_0 \int_{\VV^{d+1}} f(x,t) g(x,t) e^{-t} \d x \d t, \qquad b_0 = \frac{1}{\Gamma(d+1)} 
$$
For $f \in L^2(\VV^{d+1}, e^{-t})$, its Fourier orthogonal series on the cone is given by
$$
  f = \sum_{n=0}^\infty \sum_{m=0}^n \sum_{j=0}^{ \lfloor \f m 2 \rfloor} \sum_{\ell =1}^{\dim \CH_{m-2j}^d}
        \wh f_{m, j, \ell}^n Q_{m,j,\ell}^{n,0} \quad \hbox{with}\quad
          \wh f_{m, j, \ell}^n = \frac{ \la f, Q_{m,j,\ell}^{n,0}\ra_0}{h_{m,n}},
$$ 
where $h_{m,n}$ is given by \eqref{eq:hnm-norm}. Let $\|f\|_2$ be the norm of $f$ in $L^2(\VV^{d+1},e^{-t})$. 
By the Parseval's identity, 
$$
  \|f\|_2 =  \sum_{n=0}^\infty  \sum_{m=0}^n \sum_{j=0}^{\lfloor \frac{m}{2} \rfloor} \sum_\ell 
               |\wh f_{m,j,\ell}^m|^2 h_{m,n}.
$$
We will need to assume that $f$ is smooth so that 
\begin{equation}\label{eq:f-bound}
 \sum_{n=0}^\infty  \sum_{m=0}^n m \sum_{j=0}^{\lfloor \frac{m}{2} \rfloor} \sum_\ell 
               |\wh f_{m,j,\ell}^m|^2 h_{m,n} < \infty,
\end{equation}
which holds under a moderate assumption on $f$; for example, if $f$ is $C^2$ in $x$-variables. To be more 
precise, we let $\CD$ be the differential operator defined by 
$$
    \CD = t^2 \Delta_x - \la x,\nabla_x \ra^2 -d \la x,\nabla_x \ra, 
$$
where all derivatives acting on $x$ variables. We define a subspace of $L^2(\VV^{d+1}, e^{-t})$ by
$$
   \CW^2(\VV^{d+1},e^{-t}) := \left\{ f \in L^2(\VV^{d+1}, e^{-t}): (-\CD)^{1/2} f \in L^2(\VV^{d+1},e^{-t})\right\}.
$$
The operator $- \CD$ is positive since, as shown in \cite[(3.9)]{X}, for all $Q_{m,j,\ell}^{n,0}$,
$$
   \CD Q_{m,j,\ell}^{n,0}  = - m (m+d) Q_{m,j,\ell}^{n,0}. 
$$
In particular, the Parseval's identity implies, for $f \in  \CW^2(\VV^{d+1},e^{-t})$, that
$$
   \|(-\CD)^{1/2} f \|_2 =  \sum_{n=0}^\infty  \sum_{m=0}^n \sqrt{m(m+d)}
     \sum_{j=0}^{\lfloor \frac{m}{2} \rfloor} \sum_\ell 
               |\wh f_{m,j,\ell}^m|^2 h_{m,n}
$$
is bounded, which clearly implies \eqref{eq:f-bound}. 

\begin{thm}
Let $f \in \CW^2(\VV^{d+1}, e^{-t})$. Then the wave equation 
\begin{equation}\label{eq:waveeqn}
        \left(\partial_{tt} - c^2 \Delta_x\right) U(x,t) = e^{-t} f(x,t), 
\end{equation} 
has a solution $U(x,t) = e^{-t} u(x,t)$, where $u \in L^2(\VV^{d+1}, e^{-t})$ is given by 
\begin{align} \label{eq:waveSolu}
   u(x,t) \,& = \sum_{n=0}^\infty \sum_{m=0}^n \sum_{j=0}^{ \lfloor \f m 2 \rfloor}\sum_{\ell =1}^{\dim \CH_{m-2j}^d} 
      u_{m,j,\ell}^n Q _{m,j,\nu}^{n,-1}(x,t),
 \end{align}
where the coefficients $u_{m,j,\ell}^n$ are determined by,  
\begin{equation}\label{eq:coeff_u}
  u_{m,j,\ell}^n = \frac{1}{a_{m,j}} \sum_{i=0}^{\lfloor \frac{n-m}{2} \rfloor} \frac{(n-m)!}{(n-m-2i)!} 
      \wh f_{m+2i,j+i,\ell}^n.
\end{equation}
Furthermore, if $u \in W_2^1(\VV^{d+1})$, defined in \eqref{eq:W21}, then the solution is unique. 
\end{thm}

\begin{proof}
We first assume that $u$ is well-defined and given by \eqref{eq:waveSolu} and prove that $u$ is a solution 
of \eqref{eq:waveeqn} when the coefficients are given by \eqref{eq:coeff_u}. For convenience, we let 
$\CL: = e^{t}  (\partial_{tt} - \Delta_x) e^{-t}$. Let $u_n(x,t)$ denote the $n$-th partial sum of $u(x,t)$. 
By  \eqref{eq:waveCone}, we obtain that
\begin{align*}
 \CL u_n  = \sum_{m=0}^n \sum_{j=0}^{ \lfloor \f m 2 \rfloor} \sum_{\ell =1}^{\dim \CH_{m-2j}^d} 
       u_{m, j, \ell}^na_{m,j}  \left(Q_{m,j,\ell}^{n, 0}(x,t) - b_{m,n} Q_{m-2,j-1,\ell}^{n, 0}(x,t)\right).
\end{align*}
Reshuffling the summation, we obtain
$$ 
\CL u_n = \sum_{m=0}^n \sum_{j=0}^{ \lfloor \f m 2 \rfloor} \sum_{\ell =1}^{\dim \CH_{m-2j}^d}
   \left(a_{m,j}   u_{m, j, \ell}^n - a_{m+2,j+1} b_{m+2,n} u_{m+2,j+1,\ell}^{n} \right) Q_{m,j,\ell}^{n,0} (x,t),
$$
where we observe that $b_{m+2, n} =0$ for $m =n-1$ and $m=n$. Consequently, let 
$$
S_n f = \sum_{m=0}^n \sum_{j=0}^{ \lfloor \f m 2 \rfloor} \sum_{\ell =1}^{\dim \CH_{m-2j}^d}
        \wh f_{m, j, \ell}^n Q_{m,j,\ell}^{n,0} 
$$
denote the $n$-th partial sum of $f$; then  $\CL u_n = S_n f$ leads to
\begin{equation} \label{eq:recur}
   a_{m,j}  u_{m, j, \ell}^n -a_{m+2,j+1} b_{m+2,n} u_{m+2,j+1,\ell}^{n}  = \wh f_{m, j, \ell}^n, \quad 
   0 \le j \le \tfrac{m}2, \quad 0 \le m \le n, 
\end{equation}
for each $1 \le \ell \le \dim\CH_{m-2j}^d$. If $m =n-1$ or $m=n$, then $ b_{m+2,n} =0$. We obtain 
\begin{equation} \label{eq:coeff1}
 u_{n-1, j, \ell}^n = \frac{ \wh f_{n-1, j, \ell}^n}{a_{n-1,j}}, \quad 0\le j \le \tfrac{n-1}{2} \quad \hbox{and} \quad  
  u_{n, j, \ell}^n =  \frac{ \wh f_{n, j, \ell}^n} {a_{n,j}}, \quad 0\le j \le \tfrac{n}{2}.
\end{equation}
Furthermore, for $m = n-2, n-3,\ldots, 1, 0$, we can deduce $u_{m,j,\ell}^n$ recursively from rewriting 
\eqref{eq:recur} as 
\begin{equation} \label{eq:coeff2}
 u_{m, j, \ell}^n =\frac{1}{a_{m,j} } (\wh f_{m, j,\ell}^n +a_{m+2,j+1} b_{m+2,n} 
      u_{m+2,j+1,\ell}^{n}), \quad 0 \le j \le \tfrac{m}{2}.
\end{equation}
Notice that $\dim \CH_{m-2j}^n = \dim \CH_{(m+2)-2 (j+1)}^n$, so that $\ell$ remains fixed for its whole range,
and $0 \le j \le m/2$ is the same as $1 \le j +1 \le (m+2)/2$. It follows that $u_{m,j,\ell}^n$ are completely 
determined by \eqref{eq:coeff1} and \eqref{eq:coeff2} for each fixed $n$. Furthermore, using induction if necessary, 
the recursive relation can be solved explicitly, which is the formula for $u_{m,j,\ell}^n$ given in \eqref{eq:coeff_u}. 
Thus, we have shown that $\CL u_n = S_n f$ for all $n$. Consequently, we conclude that $\CL u = f$ in  
$L^2(\VV^{d+1}, e^{-t})$. Furthermore, if $f =0$, then all $\wh f_{m,j,\ell}^n$ are zero and \eqref{eq:coeff_u} shows
that all $u_{m,j,\ell}^n$ are zero, so that $u = 0$. Thus, the solution is unique. 
 
Next, we show that $u\in L^2(\VV^{d+1}, e^{-t})$. Using Proposition \ref{prop:Q-1toQ0}, we can write $u$ 
in terms of $Q_{m,j,\ell}^{n,0}$ as 
\begin{align*}
 & u = \sum_{n=0}^\infty  \sum_{m=0}^n \sum_{j=0}^{\lfloor \frac{m}{2} \rfloor}  
        \sum_\ell \left[ a_{m,j}  \Big( u_{m,j,\ell}^n -  2 u_{m,j,\ell}^{n+1}+ u_{m,j,\ell}^{n+2})   \right.  \\
         & \left. - a_{m+2,j+1} \Big( b_{m+2,n} u_{m+2,j+1,\ell}^n -2 c_{m+2,n+1} 
             u_{m+2,j+1,\ell}^{n+1} +d_{m+2,n+2})u_{m+2,j+1,\ell}^{n+2}\Big)   
             \right]Q_{m,j,\ell}^{n,0}.
\end{align*}
By the relation \eqref{eq:recur}, we can further deduce that 
\begin{align*}
  u = & \sum_{n=0}^\infty  \sum_{m=0}^n \sum_{j=0}^{\lfloor \frac{m}{2} \rfloor}  
        \sum_\ell \left[ \Big( f_{m,j,\ell}^n -  2 f_{m,j,\ell}^{n+1}+ f_{m,j,\ell}^{n+2})   \right.  \\
         & \left. + a_{m+2,j+1}(2m+d) \Big(- 2(n-m) u_{m+2,j+1,\ell}^{n+1} +(2n+d+3) u_{m+2,j+1,\ell}^{n+2}\Big)   
             \right]Q_{m,j,\ell}^{n,0},
\end{align*}
where we have used the relations
$$
  c_{m+2,n+1}- b_{m+2,n+1} = (2m+d)(n-m), \quad d_{m+2,n+2}- b_{m+2,n+2} = (2m+d)(2n+d+3).
$$
Thus, by the orthogonality of $Q_{m,j,\nu}^{n,0}$, the boundedness of $\|f\|_2$ and $a_{m,j} \le 1$, 
we see that $\|u\|_2$ is bounded if 
$$
 \sum_{n=0}^\infty  \sum_{m=0}^n \sum_{j=0}^{\lfloor \frac{m}{2} \rfloor}  
   \sum_\ell \left |(2m+d) \Big(- 2(n-m) u_{m+2,j+1,\ell}^{n+1} +(2n+d+3) u_{m+2,j+1,\ell}^{n+2}\Big) \right|^2 h_{m,n}  
$$ 
is bounded, which is in turn bounded by, after shifting the indices, 
\begin{align*}
 &\sum_{n=1}^\infty  \sum_{m=2}^n \sum_{j=1}^{\lfloor \frac{m}{2} \rfloor}  
   \sum_\ell \left |(2m+d-4) (n-m-1) u_{m,j,\ell}^{n}\right|^2 h_{m-2,n-1} \\
 & + \sum_{n=2}^\infty  \sum_{m=2}^n \sum_{j=1}^{\lfloor \frac{m}{2} \rfloor}  
   \sum_\ell \left |(2m+d-4)(2n+d-1) u_{m,j,\ell}^{n}\right|^2 h_{m-2,n-2}. 
\end{align*}
 By the Cauchy inequality and $a_{j,m} \ge 1/2$, we see that $u_{m,j,\ell}^n$ defined in \eqref{eq:coeff_u} satisfies 
\begin{align*}
|u_{m,j,\ell}^n|^2  \le 
  2 \sum_{i=0}^{\lfloor \frac{n-m}{2}\rfloor} \frac{(n-m)!^2}{(n-m-2i)!^2} \frac{1}{h_{m+2i,n}}  \sum_{i=0}^{\lfloor \frac{n-m}{2}\rfloor}  \left |\wh f_{m+2i,j+1,\ell}^n \right |^2  h_{m+2i,n}. 
\end{align*} 
Consequently, putting these estimates together, we conclude that 
\begin{align*}
   \|u\|_2^2 \le c_0 \|f\|_2 + c A_{m,n} \sum_{n=1}^\infty  \sum_{m=0}^n 
        \sum_{j=0}^{\lfloor \frac{m}{2} \rfloor}   \sum_{i=0}^{\lfloor \frac{n-m}{2}\rfloor}  \frac{m(m+2 i)}{n^2}
            \left |\wh f_{m+2i,j+i,\ell}^n \right |^2  h_{m+2i,n}, 
\end{align*}
where, using the explicit formula of $h_{m,n}$ in \eqref{eq:hnm-norm}, the quantity $A_{m,n}$ is 
given by 
$$
  A_{m,n} =\sum_{i=0}^{\lfloor \frac{n-m}{2}\rfloor} \frac{(n-m)!}{(n-m-2i)! (n+m)_{2i}}
       = \sum_{i=0}^{\lfloor \frac{n-m}{2}\rfloor}  \frac{(-(n-m))_{2i}}{(n+m)_{2i}}. 
$$
We show that $A_{m,n}$ is in fact bounded. Indeed, we observe that 
$$
  A_{m,n}  =  {}_2F_1 \left( \begin{matrix} -(n-m), 1 \\ n+m \end{matrix}; 1 \right) + 
     {}_2F_1 \left( \begin{matrix} -(n-m), 1 \\ n+m \end{matrix}; -1 \right). 
$$
By the Chu--Vandermond identity \cite[(15.4.24)]{DLMF}, 
$$
{}_2F_1 \left( \begin{matrix} -(n-m), 1 \\ n+m \end{matrix}; 1 \right) = \frac{(n+m-1)_{n-m}}{(n+m)_{n-m}}
 = \frac{n+m-1}{2n-1} \le 1.
$$
By the  Pfaff transformation \cite[(15.8.1)]{DLMF} and Euler's integral identity \cite[(15.6.1)]{DLMF}, we obtain
\begin{align*}
  {}_2F_1 \left( \begin{matrix} -(n-m), 1 \\ n+m \end{matrix}; - 1 \right) 
     & = 2^{-1} {}_2F_1 \left( \begin{matrix} 2n, 1, \\ n+m \end{matrix}; \frac12 \right)\\
     & = \frac{\Gamma(n+m)}{\Gamma(n+m-1)}
    \int_0^1 (1- t)^{n+m-2} (1-t/2)^{-1} dt \le 2,
\end{align*}
since $1/(1-t/2) \le 2$ on $[0,1]$. Hence, $A_{m,n}$ is bounded by 3.  Thus, using 
\begin{align*}
  \sum_{n=1}^\infty
     \sum_{m=0}^n  \sum_{j=0}^{\lfloor \frac{m}{2} \rfloor}   \sum_{i=0}^{\lfloor \frac{n-m}{2}\rfloor}  \sum_\ell 
       F_{m+2i,j+1,\ell}^n   
 = \sum_{n=1}^\infty \sum_{i=0}^{\lfloor \frac{n}{2}\rfloor}
     \sum_{m=2i}^n  \sum_{j=i}^{\lfloor \frac{m}{2} \rfloor} \sum_\ell  F_{m,j,\ell}^n,
\end{align*}
we conclude that 
$$
 \|u\|^2 \le c_0 \|f\|^2 + 3 c \sum_{n=2}^\infty   \sum_{m=0}^n  m 
      \sum_{j=0}^{\lfloor \frac{m}{2}  \rfloor} \sum_\ell  \left |\wh f_{m,j,\ell}^n \right |^2  h_{m,n}, 
$$
which is bounded since $\|(-\CD)^{1/2} f\|_2$ is finite. 

Finally, we prove that the solution of the wave equation is unique if $u \in W_2^1(\VV^{d+1})$. We only need to show that
if $f(x,t) = 0$ and $u\in W_2^1(\VV^{d+1})$, then the solution is zero. The assumption on $u$ shows, by Proposition 
\ref{prop:SOP}, that $u$ has a Fourier orthogonal series of the form \eqref{eq:waveSolu} since the usual 
Parseval identity holds in $W_2^1(\VV^{d+1})$. Consequently, by \eqref{eq:coeff_u}, $u$ is identically zero. 
\end{proof}

A couple of remarks are in order. If $f$ is a smooth function, say, in the Sobolev space 
$W_2^r(\VV^{d+1})$ consisting of functions whose $r$-th derivatives are in $L^2(\VV^{d+1})$, then
$u$ in \eqref{eq:waveSolu} is an element in $W_2^{r+2}(\VV^{d+1})$. 

Our result holds for fairly generic function $f$. For example, it applies to the function $f$ given by 
$$
  f(x,t) = g(x,t) \chi_{[a,b]}(t), \qquad 0 \le a < b < \infty,
$$
which is supported on the compact set $\{(x,t): \|x\| \le t, \, a \le t \le b\}$. Since the derivatives in $\CD$ 
applies only on $x$ variable, the theorem applies if we assume $g \in \CW^2(\VV^{d+1})$, so that 
\eqref{eq:waveSolu} gives a solution for the non-homogeneous wave equation that is equal to
$e^{-t} f(x,t)$. 

Finally, it is worth pointing out that if $f$ is a polynomial, then the solution $u$ for $L u = f$ is a 
polynomial given explicitly by \eqref{eq:waveSolu}. We end the paper with one example for $d =2$. 

In the spherical coordinates $(x_1,x_2) = r (\cos \t,\sin\t)$, the spherical harmonics in $\CH_n^2$ are just 
trigonometric functions $r^n \cos n \t$ and $r^n \sin n\t$. Hence, in the coordinates $(t,x_1,x_2) =(t, t r \cos \t, t r \sin \t)
\in \VV^3$, with $0 \le r \le 1$ and $0 \le \t \le 2 \pi$, the orthogonal basis in \eqref{eq:basisCone2} is given by 
\begin{align}\label{eq:basisCone3}
\begin{split}
  Q_{m, j,1}^{n,\mu}(x,t) & = L_{n-m}^{2 m + 2\mu + d}(t) t^{m}  P_j^{(\mu, m-2j+\f{d-2}{2})} 
       \left(2 \frac{\|x\|^2}{c^2 t^2}-1\right) r^{m-2j} \cos (m-2j) \t, \\
  Q_{m, j,2}^{n,\mu}(x,t) & = L_{n-m}^{2 m + 2\mu + d}(t) t^{m}  P_j^{(\mu, m-2j+\f{d-2}{2})} 
       \left(2 \frac{\|x\|^2}{c^2 t^2}-1\right) r^{m-2j} \sin (m-2j) \t.      
\end{split}
\end{align}
As an example that illustrates our theorem, we choose $f(x_1,x_2,t) = t x_1^2 + t^2 x_2 + x_1 x_2^2$. Its expansion in
terms of the basis $\{Q_{m, j,i}^{n,0}\}$ is given by
\begin{align*}
  f = & -\frac32 Q_{0,0,1}^{3,0}  + \frac13 Q_{1,0,1}^{3,0} + 2 Q_{1,0,2}^{3,0}
  - \f12 Q_{2,0,1}^{3,0}  - \f14 Q_{2,1,1}^{3,0}- \f14 Q_{3,0,1}^{3,0}
 + \f{1}{12} Q_{3,1,1}^{3,0} \\
 & + \f{15}{2} Q_{0,0,1}^{2,0}
 + 2 Q_{1,0,1}^{2,0}-12 Q_{1,0,2}^{2,0} + \frac72 Q_{2,0,1}^{2,0} \f74 Q_{2,1,0}^{2,0} \\
  &  - 15 Q_{0,0,1}^{1,0}+5 Q_{1,0,1}^{1,0}+30 Q_{1,0,2}^{1,0}+15 Q_{0,0,1}^{0,0}.
\end{align*}
Using \eqref{eq:waveSolu} with the coefficients determined by \eqref{eq:coeff_u}, we derive 
\begin{align*}
 u = & -3 Q_{0,0,1}^{3,-1}  + \frac12 Q_{1,0,1}^{3,-1} + 2 Q_{1,0,2}^{3,-1}
  - \f12 Q_{2,0,1}^{3,-1}  - \f12 Q_{2,1,1}^{3,-1}- \f14 Q_{3,0,1}^{3,-1}
 + \f{1}{8} Q_{3,1,1}^{3,-1} \\
 & + 11 Q_{0,0,1}^{2,-1} - 2 Q_{1,0,1}^{2,-1}-12 Q_{1,0,2}^{2,-1} + \frac72 Q_{2,0,1}^{2,-1} \f72 Q_{2,1,0}^{2,-1} \\
  &  - 15 Q_{0,0,1}^{1,-1}+5 Q_{1,0,1}^{1,-1}+30 Q_{1,0,2}^{1,-1}+15 Q_{0,0,1}^{0,-1},
\end{align*}
which, after simplification, gives 
$$
  u(t, x_1,x_2) = 8 + 2 x_1^2 + 6 x_2 + t^2 x_2 + t (2 + x_1^2 + 4 x_2) + x_1 (2 + x_2^2).
$$
One can verify right away that indeed $(\partial_{tt} - \Delta^{(x)}) e^{-t} u = e^{-t} f$.

\end{document}